\documentclass{article}
\usepackage[top=1in, bottom=1.25in, left=1in, right=1in]{geometry}
\usepackage{amsmath,amsthm,amssymb,mathtools,aliascnt,enumerate,epsfig,color,graphicx,epstopdf}
\usepackage[ruled, linesnumbered,vlined]{algorithm2e}

\usepackage{array}% http://ctan.org/pkg/array
\newcolumntype{C}[1]{>{\centering\arraybackslash$}p{#1}<{$}}

\def\<{\langle}
\def\>{\rangle}

\newcommand{\po}{\preccurlyeq}
\newcommand{\dd}{\boldsymbol{\mathrm d}}
\newcommand{\cc}{\boldsymbol{\mathrm c}}

\usepackage{hyperref}
\hypersetup{
    colorlinks=true,
    linkcolor=blue,
    citecolor=green,
    filecolor=magenta,
    urlcolor=blue
}

\newcommand{\myref}[2]{\hyperref[#1]{#2~\ref*{#1}}}

\title{The root extraction problem for generic braids}
\author{Mar\'ia Cumplido, Juan Gonz\'alez-Meneses and Marithania Silvero\footnote{Authors partially supported by the Spanish research project MTM2016-76453-C2-1-P and FEDER. First author was also supported by EPSRC New Investigator Award EP/S010963/1. Third author was also supported by the Basque Government grant IT974-16.}}
\date{August, 2019}

\begin{document}

\maketitle

% ------- Theorem styles -------
\theoremstyle{plain}
\newtheorem{theorem}{Theorem}

\newaliascnt{lemma}{theorem}
\newtheorem{lemma}[lemma]{Lemma}
\aliascntresetthe{lemma}
\providecommand*{\lemmaautorefname}{Lemma}

\newaliascnt{proposition}{theorem}
\newtheorem{proposition}[proposition]{Proposition}
\aliascntresetthe{proposition}
\providecommand*{\propositionautorefname}{Proposition}

\newaliascnt{corollary}{theorem}
\newtheorem{corollary}[corollary]{Corollary}
\aliascntresetthe{corollary}
\providecommand*{\corollaryautorefname}{Corollary}

\newaliascnt{conjecture}{theorem}
\newtheorem{conjecture}[conjecture]{Conjecture}
\aliascntresetthe{conjecture}
\providecommand*{\conjectureautorefname}{Conjecture}

\theoremstyle{remark}

\newaliascnt{claim}{theorem}
\newaliascnt{remark}{theorem}
\newtheorem{claim}[claim]{Claim}
\newtheorem{remark}[remark]{Remark}
\newaliascnt{notation}{theorem}
\newtheorem{notation}[notation]{Notation}
\aliascntresetthe{notation}
\providecommand*{\notationautorefname}{Notation}

\aliascntresetthe{claim}
\providecommand*{\claimautorefname}{Claim}

\aliascntresetthe{remark}
\providecommand*{\remarkautorefname}{Remark}

\newtheorem*{claim*}{Claim}
\theoremstyle{definition}

\newaliascnt{definition}{theorem}
\newtheorem{definition}[definition]{Definition}
\aliascntresetthe{definition}
\providecommand*{\definitionautorefname}{Definition}

\newaliascnt{example}{theorem}
\newtheorem{example}[example]{Example}
\aliascntresetthe{example}
\providecommand*{\exampleautorefname}{Example}

\def\sectionautorefname{Section}
\def\subsectionautorefname{Subsection}
\def\definitionautorefname{Definition}
\def\algorithmautorefname{Algorithm}

%|<------------------------------------------------------------------------>|

\begin{abstract}
We show that, generically, finding the $k$-th root of a braid is very fast.  More precisely, we provide an algorithm which, given a braid $x$ on $n$ strands and canonical length $l$, and an integer $k>1$, computes a $k$-th root of $x$, if it exists, or guarantees that such a root does not exist. The generic-case complexity of this algorithm is $O(l(l+n)n^3\log n)$. The non-generic cases are treated using a previously known algorithm by Sang-Jin Lee.
\end{abstract}

\section{Introduction}

There are several computational problems in braid groups that have been proposed for their potential applications to cryptography \cite{DehornoySurvey}. Initially, the conjugacy problem in the braid group~$\mathbb B_n$ was proposed as a non-commutative alternative to the discrete logarithm problem \cite{AnshelAnshelGoldfeld, Coreanos}. Later, some other problems were proposed, including the $k$-th root extraction problem: given $x\in \mathbb B_n$ and an integer $k>1$, find $a\in \mathbb B_n$ such that $a^k=x$.

The interest of braid groups for cryptography has decreased considerably, mainly due to the appearance of algorithms which solve the conjugacy problem extremely fast in the generic case \cite{Gebhardt, GebhardtGM-Sliding1, GebhardtGM-Sliding2}. The main problem with the proposed cryptographic protocols turns out to be the key generation. Public and secret keys are chosen `at random', and this implies that the protocols are insecure against algorithms which have a fast generic-case complexity.

While the future of braid-cryptography depends on finding a good key-generation procedure, there are some other problems in braid groups whose generic-case complexity is still to be studied. This is the case of the $k$-th root (extraction) problem.

A priori, the study of the generic case for the $k$-th root problem could be though to be nonsense as, generically, the $k$-th root of a braid~$x$ does not exist. But we should think of the braid~$x$ as the $k$-th power of a generic braid: in protocols based on this problem, a secret braid~$a$ is chosen at random, and the braid $x=a^k$ is made public. Hence we are dealing with braids for which a $k$-th root is known to exist. In any case, the algorithm in this paper not only shows that root extraction in braid groups is generically very fast, but can also be used by those mathematicians needing a simple algorithm for finding a $k$-th root of a braid (or proving that it does not exist), which works in most cases.

There are already known algorithms to solve the $k$-th root problem in braid groups and, more generally, in Garside groups \cite{Sibert, Lee}. But these algorithms can be simplified a lot in the generic case, as we will show in this paper.

The plan of this paper is as follows. In \autoref{preliminaries} we provide the necessary tools to describe the situation and attack the problem. Then in \autoref{section_root}, we prove the theoretical results needed for our proposed algorithm, which is given in \autoref{section_algorithm}, together with the study of its generic-case complexity.

This generic-case complexity turns out to be quadratic on the canonical length~$l$ of the braid, if the number~$n$ of strands is fixed. More precisely, the generic-case complexity is $O(l(l+n)n^3\log n)$ (\autoref{maintheorem}).

%%%%%%%%%%%%%%%%%%%%%%%%%%%%%%%%%%%%%%%%%%%%%%%%%%%%%%%%%%%%%%%%%%%%%%%%%%%%%%%%%

\section{Preliminaries}\label{preliminaries}

\subsection{Garside structure of $\mathbb B_n$}
A group $G$ is said to be a \textit{Garside group} \cite{DehornoyLibro} if it admits a submonoid~$\mathcal{P}$ (whose elements are called {\it positive}) such that $\mathcal{P}\cap \mathcal{P}^{-1}=\{1\}$, and a special element $\Delta \in \mathcal{P}$, called {\it Garside element}, satisfying the following properties:

\begin{itemize}
 \item The partial order~$\po$ in~$G$ defined by $a \po b$ if $a^{-1}b \in \mathcal{P}$ is a lattice order. If $a\po b$ we say that~$a$ is a {\it prefix} of~$b$. The lattice structure implies that for all $a,b\in G$ there exists a unique meet $a \wedge b$ and a unique join $a \vee b$ with respect to~$\po$.  Notice that this partial order is invariant under left-multiplication.

\item The set of simple elements $\mathcal{S}\coloneqq\{s\in G\,|\, 1\po s \po \Delta  \}$ is finite and generates~$G$.

\item Conjugation by $\Delta$ preserves $\mathcal{P}$, that is, $\Delta^{-1}\mathcal{P} \Delta = \mathcal{P}$.

\item $\mathcal{P}$ is atomic: the atoms are the indivisible elements of~$\mathcal P$ (elements $a\in \mathcal{P}$ for which there is no decomposition $a=bc$ with non-trivial elements $b,c\in\mathcal{P}$). Then, for every $x\in\mathcal{P}$ there is an upper bound on the number of atoms in a decomposition of the form $x=a_1a_2\cdots a_n$, where each~$a_i$ is an atom.
\end{itemize}

One of the main examples of Garside groups is the braid group on $n$~strands, denoted by~$\mathbb{B}_n$. This group has a standard presentation due to Artin \cite{Artin}:
$$
\mathbb{B}_n = \left< \sigma_1, \sigma_2, ... , \sigma_{n-1} \left|\begin{array}{cccc}
\sigma_i\sigma_j\sigma_i = \sigma_j\sigma_i\sigma_j & & &  \mbox{if }  |i-j| = 1 \\
\sigma_i\sigma_j = \sigma_j\sigma_i & & &  \mbox{if }  |i-j| > 1
\end{array}
\right.
\right>.
$$

Attending to the above presentation, a braid is said to be positive if it can be written as a product of positive powers of the generators~$\{\sigma_i\}_{i=1}^n$. The set of positive braids forms the monoid~$\mathcal{P}$ corresponding to the classical Garside structure of~$\mathbb{B}_n$. We will denote this monoid by~$\mathbb{B}_n^+$.

The usual Garside element in~$\mathbb{B}_n^+$, which we denote~$\Delta_n$, is defined recursively setting $\Delta_2=\sigma_1$ and
$$\Delta_n = \Delta_{n-1} \sigma_{n-1}\sigma_{n-2} \cdots \sigma_1,$$
for all $n>2$. We will often write~$\Delta$ and omit the subindex~$n$ when there is no ambiguity.

Consider now the inner automorphism $\tau: \mathbb{B}_n \to \mathbb{B}_n$ determined by~$\Delta$. That is, $\tau(x) = \Delta^{-1}x\Delta$. One can easily show from the presentation of~$\mathbb B_n$ that $\tau(\sigma_i) = \sigma_{n-i}$ for $1 \leq i \leq n-1$. Hence~$\tau$ has order~2 and~$\Delta^2$ is central. In fact, the center of~$\mathbb B_n$ is cyclic, generated by~$\Delta^2$ \cite{Chow}.

The set~$\mathcal S$ of simple elements and the automorphism~$\tau$ will be very important in the sequel.

\subsection{Normal forms, cyclings and decyclings}

It is well-known that Garside groups have solvable word problem, as one can compute a normal form for each element.

Let us first define the \textit{right complement} of a simple element $s\in \mathcal S$ as $\partial(s) = s^{-1} \Delta$. That is, $\partial(s)$~is the only element $t\in \mathcal P$ such that $st=\Delta$. Let us see that $\partial(s)=t$ is also a simple element. Recall that the simple elements are the positive prefixes of~$\Delta$. Since~$\tau$ preserves~$\mathcal P$ (by definition of Garside group), we have that~$\tau(s)$ is positive. Now
$$
  st\tau(s)=\Delta \tau(s) = s \Delta,
$$
hence $t\tau(s)=\Delta$, which implies that~$t$ is a positive prefix of~$\Delta$, that is, $t\in \mathcal S$. It follows that we have a map $\partial:\: \mathcal S \to \mathcal S$. Notice that, by definition, $\partial^2\equiv \tau$.

Given two simple elements $s, t \in \mathcal S$, we say that the decomposition~$st$ is \textit{left weighted} if~$s$ is the biggest possible simple element (with respect to~$\po$) in any decomposition of the element~$st$ as a product of two simple elements. This condition can be restated as $\partial(s) \wedge t = 1$, i.e., $\partial(s)$ and~$t$ have no non-trivial prefixes in common.

\begin{definition}{\rm \cite{ElrifaiMorton, Epsteinetal}}
The \emph{left normal form} of an element $x \in \mathbb{B}_n$ is the unique decomposition $x = \Delta^p x_1 \cdots x_l$ so that $p \in \mathbb{Z}$, $l \geq 0$, $x_i \in \mathcal S\setminus\{1, \Delta\}$ for $i = 1, \ldots, l$, and $x_i x_{i+1}$ is a left weighted decomposition, for $i = 1, \ldots, l-1$.
\end{definition}

Given such a decomposition, we define the \textit{infimum, supremum} and \textit{canonical length} of~$x$ as $\inf(x) = p$, $\sup(x) = p+l$ and $\ell(x) = l$, respectively. Equivalently, the infimum and supremum of~$x$ can be defined as the maximum and minimum integers~$p$ and~$s$ so that $\Delta^p \po x \po \Delta^s$ (see \cite{ElrifaiMorton}).

It is important to notice that conjugation by~$\Delta$ preserves the Garside structure of~$\mathbb B_n$. Hence, if the left normal form of a braid~$x$ is $\Delta^p x_1\cdots x_l$, then the left normal form of~$\tau(x)$ is $\Delta^p \tau(x_1)\cdots \tau(x_l)$. We will make use of this property later.

%The next result show us how to construct the normal form of $x^{-1}$:
%
%\begin{proposition}[ref]
%If $x = \Delta^p x_1 \cdots x_l$ is the left normal form of $x$, then the left normal form of $x^{-1}$ is given by the decomposition $x^{-1} = \Delta^{-p-l} x'_l \cdots x'_1$, where $x'_i = \tau^{-p-i}(\partial(x_i))$, for $i = 1, \ldots, l$. In particular, $\inf(x^{-1}) = -\inf(x)-\ell(x)$, $\sup(x^{-1}) = -\inf(x)$ and $\ell(x^{-1})= \ell(x)$.
%\end{proposition}
%
%\bigskip

Garside groups also have solvable conjugacy problem. One of the main tools to solve problems related to conjugacy in braid groups are the summit sets, which are subsets of the conjugacy class of a braid. Throughout this article we are going to use two of them: the {\it super summit set} \cite{ElrifaiMorton} and the {\it ultra summit set} \cite{Gebhardt}. Let us first introduce some concepts:

\begin{definition}
Let $x = \Delta^p x_1 \cdots x_l$ be in left normal form, with $l>0$. Notice that we can write:
$$
 x= \tau^{-p}(x_1) \Delta^p x_2\cdots x_l.
$$
We define the {\it initial factor} of $x$ as $\iota(x) =\tau^{-p}(x_1)$, and the {\it final factor} of $x$ as  $\varphi(x) = x_l$. We can then write:
$$
 x = \iota(x)\Delta^p x_2 \cdots x_l  \, \quad \mbox{ and } \quad \, x =  \Delta^p x_1 \cdots x_{l-1}\varphi(x).
$$
If $l=0$, we set $\iota(x) = 1$ and $\varphi(x) = \Delta$.
\end{definition}

Notice that, as~$\tau^2$ is the identity, we actually have either $\iota(x)=x_1$ if~$p$ is even, or $\iota(x)=\tau(x_1)$ if~$p$ is odd. This happens in braid groups, but not in other Garside groups in which the order of~$\tau$ is bigger.

\begin{definition} {\rm{\cite{ElrifaiMorton}}}
Let $x = \Delta^p x_1 \cdots x_l$ be in left normal form, with $l>0$. The \textit{cycling} and \textit{decycling} of~$x$ are the conjugates of~$x$ defined, respectively, as
$$
 \cc(x) = \Delta^p x_2 \cdots x_l \iota(x) \, \quad \mbox{ and } \quad \, \dd(x) = \varphi(x) \Delta^p x_1 \cdots x_{l-1}.
$$
\end{definition}

Thus~$\cc(x)$ is the conjugate of~$x$ by~$\iota(x)$, and that~$\dd(x)$ is the conjugate of~$x$ by~$\varphi(x)^{-1}$.

Cyclings and decyclings were defined in \cite{ElrifaiMorton} in order to try to simplify the braid~$x$ by conjugations. Usually, if $l\geq 2$, the decomposition $\Delta^p x_2 \cdots x_l \iota(x)$ is {\bf not} the left normal form of~$\cc(x)$. So~$\cc(x)$ could a priori have a shorter normal form (with less factors). A similar situation happens for~$\dd(x)$.

If $\Delta^p x_2 \cdots x_l \iota(x)$ is actually the left normal form of~$\cc(x)$ (when $l\geq 2$), we say that the braid~$x$ is {\it rigid}. This happens if and only if $x_l\iota(x)$ (that is, $\varphi(x)\iota(x)$) is a left weighted decomposition. We can extend this definition to every case, when $l\geq 0$:

\begin{definition}
We say that $x\in \mathbb B_n$ is {\it rigid} if $\varphi(x)\iota(x)$ is a left weighted decomposition.
\end{definition}

If~$x$ is rigid, neither cycling nor decycling can simplify its normal form $x=\Delta^p x_1 \cdots x_l$. Actually, the normal forms of the iterated cyclings of~$x$ are, if~$p$ is even:
$$
    \cc(x)=\Delta^p x_2\cdots x_l x_1, \qquad \cc^2(x)=\Delta^p x_3\cdots x_lx_1x_2, \qquad \ldots
$$
so $\cc^l(x)=x$ in this case. In the case when~$p$ is odd we have:
$$
    \cc(x)=\Delta^p x_2\cdots x_l \tau(x_1), \qquad \cc^2(x)=\Delta^p x_3\cdots x_l\tau(x_1)\tau(x_2), \qquad \ldots
$$
so $\cc^{2l}(x)=x$ in this case.

In the same way, if~$x$ is rigid we have, for~$p$ even:
$$
    \dd(x)=\Delta^p x_l x_1\cdots x_{l-1}, \qquad \dd^2(x)=\Delta^p x_{l-1} x_{l} x_1\cdots x_{l-2}, \qquad \ldots
$$
so $\dd^l(x)=x$ in this case. If~$p$ is odd we get:
$$
   \dd(x)=\Delta^p \tau(x_l) x_1\cdots x_{l-1}, \qquad \dd^2(x)=\Delta^p \tau(x_{l-1}) \tau(x_{l}) x_1\cdots x_{l-2}, \qquad \ldots
$$
so $\dd^{2l}(x)=x$ in this case. We then see that, if~$x$ is rigid, iterated cyclings and decyclings correspond to cyclic permutations of the factors in the normal form of~$x$ (possibly conjugated by~$\Delta$, if~$p$ is odd); moreover, when applied to rigid braids, $\cc$ and~$\dd$ are inverses of each other.

\subsection{Summit sets}

Let now $x\in \mathbb B_n$ be an arbitrary braid (not necessarily rigid). Consider the conjugacy class of~$x$, denoted~$x^{\mathbb{B}_n}$, and write~${\inf}_s(x)$ (resp.~${\sup}_s(x)$) for the maximal infimum (resp. the minimal supremum) of an element in~$x^{\mathbb{B}_n}$. These numbers are known to exist \cite{ElrifaiMorton}, and are called the {\it summit infimum} and the {\it summit supremum} of~$x$, respectively. Set $\ell_s(x)=\sup_s(x)-\inf_s(x)$, the {\it summit length} of~$x$. It is shown in \cite{ElrifaiMorton} that the elements in~$x^{\mathbb B_n}$ having the shortest possible normal form are those whose canonical length is precisely~$\ell_s(x)$, and they coincide with the elements whose infimum and supremum are equal to~$\inf_s(x)$ and~$\sup_s(x)$, respectively. The set formed by these elements is called the \textit{supper summit set} of the braid~$x$:
$$
  SSS(x) = \left\{y \in x^{\mathbb{B}_n} \, | \, \ell(y) = {\ell}_s(x)\right\}= \left\{y \in x^{\mathbb{B}_n} \, | \, \inf(y) = {\inf}_s(x), \, \, \sup(y) = {\sup}_s(x)\right\}.
$$

Starting from~$x$, it is possible to obtain an element in~$SSS(x)$ by applying cyclings and decyclings iteratively. It is known \cite{ElrifaiMorton} that if $\inf(x)<\inf_s(x)$ then the infimum of~$x$ can be increased by iterated cycling. Actually, in this case $\inf(x)<\inf(\cc^k(x))$ for some $k<\frac{n(n-1)}{2}$ (see \cite{BirmanKoLee2}). Hence, every $\frac{n(n-1)}{2}$ iterations either the infimum has increased, or one is sure to have an element whose infimum is the summit infimum.

In the same way, if $\sup(x)>\sup_s(x)$, then the supremum of~$x$ can be decreased by iterated decycling \cite{ElrifaiMorton}, and in that case $\sup(x)>\sup(\dd^k(x))$ for some $k<\frac{n(n-1)}{2}$ \cite{BirmanKoLee2}. Hence, every $\frac{n(n-1)}{2}$ iterations either the supremum has decreased, or we are sure to have an element whose supremum is the summit supremum. Since decycling can never decrease the infimum of an element, it follows that starting with any $x\in \mathbb B_n$ and applying iterated cycling (until summit infimum is obtained) followed by iterated decycling (until summit supremum is obtained) yields an element $y\in SSS(x)$.

The super summit set~$SSS(x)$ is a finite set, but it is usually huge, so smaller subsets of the conjugacy class of~$x$ were defined in order to solve the conjugacy problem of~$x$ more efficiently. Namely, the \emph{ultra summit set} of~$x$, denoted by~$USS(x)$, is a subset of~$SSS(x)$ defined as follows \cite{Gebhardt}:
$$
   USS(x)=\{y\in SSS(x)\,|\,  \cc^m(y)=y \text{ for some } m>0\}.
$$
Since~$SSS(x)$ is finite, the subset~$USS(x)$ is also finite. It is then clear that one obtains an element is~$USS(x)$ by iterated application of cycling, starting from an element in~$SSS(x)$, when a repeated element is obtained. Actually, the whole orbit under cycling of an element in~$USS(x)$ belongs to~$USS(x)$. So~$USS(x)$ is a finite set of orbits under cycling.

Notice that every rigid braid belongs to its ultra summit set, as cylings and decyclings are basically cyclic permutations of its factors. It is shown in \cite{BirmanGebhardtGM1} that, if~$x$ is conjugate to a rigid braid and $\ell_s(x)>1$, then~$USS(x)$ coincides with the set of rigid conjugates of~$x$.

There is actually a simpler way, in the general case, to obtain an element in~$USS(x)$ starting from~$x$. Instead of using cyclings and decyclings, one can use the following single type of conjugation:

\begin{definition}\rm{\cite{GebhardtGM-Sliding1}}
Given $x\in \mathbb B_n$, the {\it cyclic sliding} of~$x$ is defined as $\mathfrak s(x)=\mathfrak p(x)^{-1} x \: \mathfrak p(x)$, where $\mathfrak p(x)=\iota(x)\wedge \partial(\varphi(x))$.
\end{definition}

\begin{theorem}\rm{\cite{GebhardtGM-Sliding1}}
Given $x\in \mathbb B_n$, there are integers $0\leq k<t$ such that $\mathfrak s^k(x)=\mathfrak s^t(x)$. For every such pair of integers, one has $\mathfrak s^k(x)\in USS(x)$.
\end{theorem}

By the above result, one can obtain an element in~$USS(x)$ by iterated cyclic sliding starting form~$x$. Furthermore, if~$x$ is conjugate to a rigid element (this will be the generic situation, as we will see in \autoref{subsection_generic}), iterated cyclic sliding yields the {\it shortest} positive conjugating element from~$x$ to a rigid element.

\begin{theorem}\label{Cyclic_Sliding_fastest_conjugation}\rm{\cite{GebhardtGM-Sliding1}}
Let $x\in \mathbb B_n$ and suppose that~$x$ is conjugate to a rigid braid. Then there is an integer $k>0$ such that~$\mathfrak s^k(x)$ is rigid. Moreover, the conjugating element~$\alpha$ from~$x$ to~$\mathfrak s^k(x)$, that is,
$$
   \alpha = \mathfrak p(x) \: \mathfrak p(\mathfrak s(x))\: \mathfrak p(\mathfrak s^2(x))\cdots \mathfrak p(\mathfrak s^{k-1}(x))
$$
is the smallest positive element (with respect to~$\po$) conjugating~$x$ to a rigid element, meaning that for every positive element~$\beta$ such that $\beta^{-1}x\beta$ is rigid, one has $\alpha\po \beta$.
\end{theorem}

After obtaining one element in~$USS(x)$, it is possible to compute all elements in~$USS(x)$ together with conjugating elements connecting them. In this way, one solves the conjugacy problem in~$\mathbb B_n$, as two elements~$x$ and~$y$ are conjugate if and only if $USS(x)=USS(y)$ or, equivalently, if $USS(x)\cap USS(y)\neq \emptyset$. Then, in order to check whether~$x$ and~$y$ are conjugate, one can compute the whole set~$USS(x)$, and one element $\widetilde y\in USS(y)$. Then, $x$ and~$y$ are conjugate if and only if $\widetilde y\in USS(x)$. By construction, one can even compute a conjugating element from~$x$ to~$y$.

In order to understand the forthcoming proofs in this paper, we will need to describe some conjugating elements connecting the elements of~$USS(x)$.

\begin{definition}\rm{\cite{Gebhardt}}
Let $x \in \mathbb{B}_n$ and $y \in USS(x)$. A simple non-trivial element $s \in \mathcal S$ is said to be a \emph{minimal simple element} for~$y$ if $y^s \in USS(x)$ and $y^t \notin USS(x)$, for every $1 \prec t \prec s$.
\end{definition}

In \cite{Gebhardt}, Gebhardt showed that for any two elements $y,z\in USS(x)$ there exists a sequence
$$
    y=y_1 \stackrel{c_1}{\longrightarrow} y_2 \stackrel{c_2}{\longrightarrow} \cdots \rightarrow y_t \stackrel{c_t}{\longrightarrow} y_{t+1} =z,
$$
where~$c_i$ is a minimal simple element for~$y_i$, and $y_{i+1}=c_i^{-1}y_ic_i$, for $i=1,\ldots,t$. Moreover, he introduced an algorithm to compute all minimal simple elements for a given $y \in USS(x)$. This allows to construct a directed graph~$\Gamma_x$, whose vertices correspond to elements of~$USS(x)$, and whose arrows correspond to minimal simple elements, in such a way that for every minimal simple element~$s$ for~$y$, there is an edge with label~$s$ from~$y$ to $y^s = s^{-1}ys$. By the above discussion, it follows that~$\Gamma_x$ is a connected graph, and this is why~$USS(x)$ can be computed starting with a single vertex, iteratively computing the minimal simple elements corresponding to each known vertex, until all vertices are obtained.

We will later see that, generically, ultra summit sets are really small. Actually, they usually have a very simple structure, that we explain now.

\begin{lemma} \rm{\cite{BirmanGebhardtGM2}}
Let $y \in USS(x)$ with $\ell(y) >0$ and let~$s$ be a minimal simple element for~$y$. Then, $s$~is a prefix of either~$\iota(y)$ or~$\partial(\varphi(y))$, or both.
\end{lemma}

The above lemma allows us to classify the arrows in~$\Gamma_x$ into two groups: a directed edge labelled by~$s$ starting at $y \in USS(x)$ is black (resp. grey), if~$s$ is a prefix of~$\iota(x)$ (resp. of~$\partial(\varphi(y))$). In principle, an edge could be of both colors at the same time (a bi-colored arrow, whose label is a prefix of both~$\iota(x)$ and~$\partial(\varphi(x))$), but not in the case of rigid braids, as $\iota(x)\wedge \partial(\varphi(x))=1$ if~$x$ is rigid. Actually, this is a necessary and sufficient condition:

\begin{lemma}\rm{\cite{BirmanGebhardtGM2}}\label{lemma_rigid_iff}
A braid $y \in USS(x)$ with $\ell(y) >0$ is rigid if and only if none of the edges starting at~$y$ is bi-colored.
\end{lemma}

\begin{definition}\label{def_USSminimal}
Given a braid $x\in \mathbb{B}_n$, its associated~$USS(x)$ is \emph{minimal} if $\ell_s(x)>1$ and, for every vertex~$y$ in the graph~$\Gamma_x$, there are exactly two directed edges starting at~$y$, a black one labeled~$\iota(y)$ and a grey one labeled~$\partial(\varphi(y))$. \end{definition}

Notice that, as a consequence of \autoref{lemma_rigid_iff}, if~$USS(x)$ is minimal then all elements in $USS(x)$ are rigid. Moreover, the arrow labeled~$\iota(y)$ corresponds to a cycling of~$y$, and the arrow labeled~$\partial(\varphi(y))$ corresponds to a {\it twisted decycling} of~$y$, meaning a decycling followed by the automorphism~$\tau$. This implies that, if~$USS(x)$ is minimal, the elements of~$USS(x)$ are obtained from~$y$ by applying~$\cc$ and~$\tau\circ \dd$ in every possible way. Since~$y$ is rigid, cyclings and decyclings basically correspond to cyclic permutations of the factors. Therefore, if~$USS(x)$ is minimal, it consists of either two orbits under cycling (conjugate to each other by~$\Delta$), or one orbit under cycling (conjugate to itself by~$\Delta$). If the infimum of~$y$ is even, the orbit of~$y$ has at most $\ell (y) = \ell_s(x)\leq \ell(x)$ elements, so the size of~$USS(x)$ is at most~$2\ell(x)$. If the infimum of~$y$ is odd, the orbit of~$y$ has at most $2\ell(y)\leq 2\ell(x)$ elements, and it is conjugate to itself by~$\Delta$, so it is the only orbit. Therefore, in any case, if~$USS(x)$ is minimal it has at most~$2\ell(x)$ elements.
\\

\begin{remark}\label{checking_minimal} In order to see whether $USS(x)$ is minimal, one should a priori check the condition in \autoref{def_USSminimal} for every element in $USS(x)$. But it is actually shown in \cite[Theorem~4.6]{GMValladares} that, given $y\in USS(x)$, the set~$USS(x)$ is minimal if and only if $\ell(y)>1$ and the minimal simple elements for~$y$ are precisely~$\iota(y)$ and~$\partial(\varphi(y))$. Hence, one just needs to compute the minimal elements for a single arbitrary element $y \in USS(x)$.
\end{remark}

Let us see that this case, in which $USS(x)$ is so small and has such a simple structure, is generic.

\subsection{Generic braids}\label{subsection_generic}

Since~$\mathbb B_n$ is an infinite set, it is necessary to explain what we mean by `picking a random braid' or by saying that a braid is `generic'. Even if we fix the subset of braids of a given length, we must specify if we choose braids from the subset with a uniform distribution, or if we pick braids by choosing a random walk in the Cayley graph, which are the two usual situations.

We will consider the Cayley graph of the braid group~$\mathbb{B}_n$, taking as generators the simple braids, and assume that each edge of the Cayley graph has length~1, so it becomes a metric space. Let us point out that left normal forms of braids are closely related to geodesics in this Cayley graph \cite{Charney}.

Now let~$B(r)$ denote the ball of radius~$r$ centered at the trivial braid~$1$. As the number of simple braids is finite, the set~$B(r)$ is a finite subset of~$\mathbb B_n$. We will consider the uniform distribution within this set. It turns out that `most' elements in~$B(r)$ have a very simple ultra summit set:

\begin{theorem}\rm{\cite{GMValladares}}
The proportion of braids in~$B(r)$ whose ultra summit set is minimal tends to~$1$ exponentially fast, as~$r$ tends to infinity.	
\end{theorem}

This is why we can say that the ultra summit set of a `generic braid' is minimal. Moreover, the above result was obtained by refining the following theorem, which gives some important information concerning the elements in~$B(r)$. We have simplified the statement to adapt it to our situation:

\begin{theorem}\rm{\cite{CarusoWiest}}\label{teoCarusoWiest}
The proportion of braids~$x$ in~$B(r)$ which are conjugate to a rigid braid $y=\alpha^{-1} x \alpha$, in such a way that~$\alpha$ is a positive braid with $\ell(\alpha)<\ell(x)$, tends to~$1$ exponentially fast, as~$r$ tends to infinity.
\end{theorem}

Therefore, not only generic braids have minimal ultra summit sets (made of rigid braids), but one can also obtain a rigid conjugate of a generic braid~$x$ very fast, applying iterated cyclic sliding to~$x$. By \autoref{Cyclic_Sliding_fastest_conjugation}, the obtained conjugating element will be the smallest possible positive conjugator, so its canonical length will be smaller than~$\ell(x)$. Once that a rigid conjugate~$y$ (which belongs to~$USS(x)$) is obtained, one can compute the whole~$USS(x)$ very fast, as it consists of at most~$2\ell(x)$ elements, connected by cyclings and twisted decyclings. This is why solving the conjugacy problem in braid groups is generically very fast.

We will also be interested in the centralizer~$Z(x)$ of a braid~$x$. Notice that if $y=\alpha^{-1}x\alpha$, then $Z(y)=\alpha^{-1} Z(x)\alpha$. Therefore, knowing~$Z(y)$ is equivalent to knowing~$Z(x)$, via~$\alpha$. We will then be interested in~$Z(y)$ for $y\in USS(x)$.

\begin{definition}\label{defPC}
Let $x \in \mathbb B_n$ and $y\in USS(x)$, and let~$t$ be the smallest positive integer such that $\cc^t(y) = y$. Denote $p_i:=\iota(\cc^{i-1}(y))$ the positive element conjugating~$\cc^{i-1}(y)$ to~$\cc^{i}(y)$, for $i=1,\ldots,t$. Then the \textit{preferred cycling conjugator} of~$y$ is defined as
\[
  PC(y) = p_1 p_2 \cdots p_t.
\]
In other words, $PC(y)$~corresponds to the conjugating element along the whole cycling orbit of~$y$. By construction, $PC(y)$ commutes with~$y$.
\end{definition}

In the generic case (when~$USS(x)$ is minimal), it turns out that~$Z(x)$ is isomorphic to~$\mathbb Z^2$, and one can describe the generators of~$Z(y)$ for any $y\in USS(x)$ (and thus of~$Z(x)$) in a very explicit way:

\begin{theorem}\rm{\cite{GMValladares}}\label{teoJuanLoles}
Let $x\in \mathbb{B}_n$ and $y\in USS(x)$. Let $PC(y)=p_1\cdots p_t$ as above. If~$USS(x)$ is minimal, then all elements in~$USS(x)$ are rigid, $Z(x)\simeq Z(y) \simeq \mathbb Z^2$, and one of the following conditions holds:
\begin{itemize}
\item[(i)] $USS(x)$ has two orbits under cycling, conjugate to each other by~$\Delta$, and $Z(y) = \langle \Delta^2, PC(y)\rangle$.
\item[(ii)] $USS(x)$ has one orbit under cycling, conjugate to itself by~$\Delta$, and:
	\begin{itemize}
		\item If $\tau(y) = y$, then $Z(y) = \langle \Delta, PC(y) \rangle$.
		\item If $\tau(y) \neq y$, then~$t$ is even and $Z(y) = \langle \Delta^2, \: p_1 \cdots p_{\frac{t}{2}}\Delta^{-1} \rangle$.
	\end{itemize}
\end{itemize}
\end{theorem}

\section{$k$-th root problem}\label{section_root}

Now we come to the central problem in this paper: given $x\in \mathbb B_n$ and an integer $k>1$, find a $k$-th root of~$x$. In other words, we want to either find $a\in \mathbb B_n$ such that $a^k=x$, or show that such a braid does not exist.

Notice that if $a^k=x$ then~$a$ belongs to~$Z(x)$, the centralizer of~$x$. It is interesting to know that finding a single solution~$a$ to the $k$-th root equation is basically the same as finding all possible solutions, as the complete set of solutions coincides with the conjugacy class of~$a$ in~$Z(x)$:

\begin{proposition}\label{roots_are_conjugate_in_Z(x)}
Let $a,x\in \mathbb B_n$ be such that $a^k=x$ for some integer $k>1$. Then the set $\sqrt[k]{x}$ of $k$-th roots of $x$ is precisely
$$
    \sqrt[k]{x}= a^{Z(x)}=\left\{b\in \mathbb B_n \, | \, \ b=u^{-1}au,\ u\in Z(x) \right\}.
$$
\end{proposition}

\begin{proof}
In \cite{GM-roots}, the second author proved that the $k$-th root of a braid is unique, up to conjugacy. That is, if $a,b\in \mathbb B_n$ satisfy $a^k=b^k=x$, then $a=u^{-1}bu$ for some $u\in \mathbb B_n$. Then one has $x= b^k = u^{-1} a^k u =u^{-1} x u$, and hence $u\in Z(x)$. This proves that $\sqrt[k]{x}\subset a^{Z(x)}$.

On the other hand, if $b=a^{Z(x)}$ and we write $b=u^{-1}au$ for some $u\in Z(x)$, we have $b^k=u^{-1}a^k u = u^{-1} x u = x$, so $b\in \sqrt[k]{x}$.
\end{proof}

Observe that $a^k=x$ if and only if  $(\alpha^{-1}a\alpha)^k=\alpha^{-1}x\alpha$ for any $\alpha\in \mathbb B_n$. Hence, given~$x$, it suffices to solve the $k$-th root problem for any conjugate of~$x$, for instance for some $y \in USS(x)$.

We will focus our attention in the generic case in which $USS(x)$ is minimal. Recall from \autoref{teoJuanLoles} that in this case $Z(x)\simeq Z(y)\simeq \mathbb Z^2$. If we express the centralizer of~$y$ as $Z(y)=\langle v,w\rangle$, where~$v$ and~$w$ commute, we know that $y$ has the form $y=v^c w^d$, for some $c,d\in \mathbb Z$ (and that this expression is unique, as any other expression would yield a different element of~$Z(y)$). If we are able to express $y$ in this way, then the $k$-th root problem is trivially solved:

\begin{proposition}\label{prop_raiz}
Let $x\in \mathbb B_n$. Let $y\in USS(x)$ and suppose that $USS(x)$ is minimal. Let $Z(y)=\langle v,w\rangle$ and let $c,d\in \mathbb Z$ be such that $y=v^c w^d$. Then~$y$ admits a $k$-th root if and only if both $c$ and $d$ are multiples of $k$, and in this case the only $k$-th root of~$y$ is:
\[
     a= v^{\frac{c}{k}} w^{\frac{d}{k}}.
\]
\end{proposition}

\begin{proof}
We know from \autoref{teoJuanLoles} that $Z(y)\simeq \mathbb Z^2$, so it is abelian. Hence, by \autoref{roots_are_conjugate_in_Z(x)}, if a $k$-th root~$a$ of~$y$ exists then $\sqrt[k]{y} = a^{Z(y)} = \{a\}$. Therefore, if a $k$-th root exists, it is unique.

Suppose that the $k$-th root problem for~$y$ has a solution $a\in \mathbb B_n$. Then $a\in Z(y)$, and hence $a=v^r w^s$ for some $r,s\in \mathbb Z$. But since~$v$ and~$w$ commute, we have:
\[
   v^c w^d = y = a^k = (v^r w^s)^k = v^{rk} w^{sk}.
\]
This implies that~$c$ and~$d$ are multiples of~$k$, and that $a= v^r w^s= v^{\frac{c}{k}} w^{\frac{d}{k}}.$

Conversely, if~$c$ and~$d$ are multiples of~$k$, we write $c=rk$ and $d=sk$ for some integers $r, s$, and we consider the element $a=v^r w^s$. Since~$v$ and~$w$ commute, it follows that $a^k=y$.
\end{proof}

By the above result, it follows that the only difficulty in solving the $k$-th root problem, in the generic case in which $USS(x)$ is minimal, is to express some $y\in USS(x)$ in terms of the generators of~$Z(y)$. We know from \autoref{teoJuanLoles} that there are three possible cases, depending on whether~$USS(x)$ has two orbits under cycling, or has one orbit with $\tau(y)=y$, or has one orbit with $\tau(y)\neq y$. The three following results address each case:

\begin{proposition}\label{case1}
Let $x\in \mathbb B_n$, and let $y=\Delta^p y_1\cdots y_l\in USS(x)$, written in left normal form. Suppose that $USS(x)$ is minimal. Suppose also that $USS(x)$ has two orbits under cycling, conjugate to each other by~$\Delta$. Let $v=\Delta^2$ and $w=PC(y)=p_1\cdots p_t$, so:
$$
    Z(y) = \langle v, w\rangle = \langle \Delta^2, PC(y)\rangle.
$$
If we write $c=p/2$ and $d=l/t$, then~$c$ and~$d$ are integers and we have: $y=v^c w^d$.
\end{proposition}

\begin{proof}
We know that, since $USS(x)$ is minimal, it consists of rigid elements. Hence iterated cycling corresponds to a cyclic permutation of the factors in the normal form of~$y$ (with possible conjugations by~$\Delta$, if~$p$ is odd).

Suppose that~$p$ is odd. Then~$\cc^l(y)$ is obtained from~$y$ by cyclically permuting all its~$l$ factors, conjugating all of them by $\Delta$. Hence $\cc^l(y)=\tau(y)$. This implies that $\tau(y)=\Delta^{-1}y\Delta$ is in the same orbit of~$y$ under cycling, but this is a contradiction with the hypotheses, as $USS(x)$ has two distinct orbits (the one containing~$y$ and the one containing~$\tau(y)$). Therefore~$p$ is even.

Since~$p$ is even, iterated cyclings of~$y$ correspond exactly to cyclic permutations of the factors of~$y$. By definition, $t$~is the smallest positive integer such that $\cc^t(y)=y$, and it is then clear that $\cc^m(y)=y$ for some positive integer~$m$ if and only if~$m$ is a multiple of~$t$. Since $\cc^l(y)=y$, we finally obtain that~$l$ is a multiple of~$t$. Then the normal form of~$y$ is as follows:
\[
    y= \Delta^p y_1\cdots y_l = \Delta^p (y_1\cdots y_t)(y_1\cdots y_t)\cdots (y_1\cdots y_t),
\]
where $PC(y)=y_1\cdots y_t$, and there are~$l/t$ parenthesized factors.

Now, if we write $c=p/2$ and $d=l/t$, these numbers are integers and we have:
\[
  v^c w^d = (\Delta^2)^c (PC(y))^d = \Delta^{2c} (y_1\cdots y_t)^d = \Delta^p y_1\cdots y_l = y.
\]
\end{proof}

\begin{proposition}\label{case2}
Let $x\in \mathbb B_n$, and let $y=\Delta^p y_1\cdots y_l\in USS(x)$, written in left normal form. Suppose that $USS(x)$ is minimal. Suppose also that $USS(x)$ has one orbit under cycling, conjugate to itself by~$\Delta$, and that $\tau(y)=y$. Let $v=\Delta$ and $w=PC(y)=p_1\cdots p_t$, so:
$$
    Z(y) = \langle v, w\rangle = \langle \Delta, PC(y)\rangle.
$$
If we write $c=p$ and $d=l/t$, then~$c$ and~$d$ are integers and we have: $y=v^c w^d$.
\end{proposition}

\begin{proof}
We know that the left normal form of $\tau(y)$ is $\Delta^p \tau(y_1)\cdots \tau(y_l)$. Since $\tau(y)=y$, the normal forms of~$y$ and~$\tau(y)$ must coincide, hence $\tau(y_i)=y_i$ for $i=1,\ldots,l$.

This implies that iterated cyclings correspond to cyclic permutations of the factors of~$y$. We do not care about the parity of~$p$, as every factor of~$y$ is invariant under $\tau$. It then follows that $PC(y)=y_1\cdots y_t$, that $t$ divides $l$ and that the normal form of $y$ is:
$$
   y = \Delta^p y_1\cdots y_l = \Delta^p (y_1\cdots y_t)(y_1\cdots y_t)\cdots (y_1\cdots y_t),
$$
where there are $l/t$ parenthesized factors.

Now, if we write $c=p$ and $d=l/t$, these numbers are integers and we have:
$$
   v^c w^d = \Delta^c (PC(y))^d = \Delta^{c} (y_1\cdots y_t)^d = \Delta^p y_1\cdots y_l = y.
$$
\end{proof}

\begin{proposition}\label{case3}
Let $x\in \mathbb B_n$, and let $y=\Delta^p y_1\cdots y_l\in USS(x)$, written in left normal form. Suppose that $USS(x)$ is minimal. Suppose also that $USS(x)$ has one orbit under cycling, conjugate to itself by $\Delta$, and that $\tau(y)\neq y$. Let $v=\Delta^2$, $PC(y) = p_1 \cdots p_t$ and $w=p_1\cdots p_{\frac{t}{2}}\Delta^{-1}$ (recall from \autoref{teoJuanLoles} that $t$ is even), so:
$$
    Z(y) = \langle v, w\rangle = \langle \Delta, \: p_1\cdots p_{\frac{t}{2}}\Delta^{-1}\rangle.
$$
If we write $c=\frac{pt+2l}{2t}$ and $d=\frac{2l}{t}$, then $c$ and $d$ are integers and we have: $y=v^c w^d$.
\end{proposition}

\begin{proof}
We know from \autoref{teoJuanLoles} that~$t$ is even, but let us see why this holds. We know that there exists some $m>0$ so that $\tau(y)=\cc^m(y)$; we take $m$ as small as possible, and this implies that $\cc^r(y)\neq y$ for $0<r<m$. Now, it follows from their own definitions that~$\tau$ and~$\cc$ commute, and therefore $y = \tau^2(y) = \tau(\cc^m(y)) = \cc^m(\tau(y))= \cc^{2m}(y)$. This implies that the length of the cycling orbit of~$y$ is a divisor of~$2m$. It cannot be~$m$ (as $\cc^m(y)=\tau(y)\neq y$), and it cannot be smaller than~$m$ (as $\cc^r(y)\neq y$ for every $r<m$). Therefore, the length of the orbit is precisely $t=2m$.  The generators of~$Z(y)$ are then $v=\Delta^2$ and $w=p_1\cdots p_m \Delta^{-1}$.

We consider now two cases, depending on the parity of~$p$. If~$p$ is even, since the first $m$~cyclings transform~$y$ into~$\tau(y)$, it follows that the left normal form of~$y$ is:
$$
    y= \Delta^p \ \left(y_1\cdots y_{m}\right)\  \left(\tau(y_1)\cdots \tau(y_{m})\right)\ \cdots \  \left(y_1\cdots y_{m}\right)\ \left(\tau(y_1)\cdots \tau(y_{m})\right).
$$
Then $l=2rm$ for some positive integer~$r$.

Recall that~$PC(y)$ is the product of the first $t=2m$ conjugating elements for cycling. The first $m$~conjugating elements are $y_1,\ldots, y_m$, so $p_i=y_i$ for $i=1,\ldots,m$. The following $m$~conjugating elements are~$\tau(y_1),\ldots, \tau(y_m)$. Hence, we have that
\begin{eqnarray*}
 PC(y) & = & p_1\cdots p_t
 \\ & = & y_1\cdots y_{m} \tau(y_1)\cdots \tau(y_{m})
 \\ & = & y_1\cdots y_{m} \;\tau(y_1\cdots y_{m})
 \\ & = & p_1\cdots p_{m} \Delta^{-1} p_1\cdots p_{m} \Delta
 \\ & = & \left(p_1\cdots p_{m} \Delta^{-1}\right) \left(p_1\cdots p_{m} \Delta^{-1}\right) \Delta^2
 \\ & = &  w^2 v.
\end{eqnarray*}

Therefore, if~$p$ is even:
$$
  y= \Delta^p PC(y)^{r} = v^{\frac{p}{2}} \left(w^2v\right)^{r} = v^{\frac{p}{2}+r} w^{2r} = v^c w^d,
$$
where $c=\frac{pt+2l}{2t}$ and $d=\frac{2l}{t}$ (recall that $l=2rm=rt$).

Consider now the case when~$p$ is odd. In this case, the left normal form of~$y$ is:
$$
    y= \Delta^p \ \left(y_1\cdots y_{m}\right)\  \left(\tau(y_1)\cdots \tau(y_{m})\right)\ \cdots \  \left(y_1\cdots y_{m}\right)\ \left(\tau(y_1)\cdots \tau(y_{m})\right)\ \left(y_1\cdots y_{m}\right).
$$
Then $l=(2r+1)m$ for some positive integer $r$.

As before, $PC(y)$ is the product of the first $t=2m$ conjugating elements for cycling, but this time the first $m$ conjugating elements for cycling are $\tau(y_1),\ldots,\tau(y_m)$, and therefore $p_i=\tau(y_i)$ for $i=1,\ldots,m$. The following $m$ conjugating elements are $y_1,\ldots,y_m$, so we have:
\begin{eqnarray*}
 PC(y) & = & p_1\cdots p_t
 \\ & = & \tau(y_1)\cdots \tau(y_{m}) y_1\cdots y_{m}
 \\ & = & p_1\cdots p_{m} \;\tau(p_1\cdots p_{m})
 \\ & = & p_1\cdots p_{m} \Delta^{-1} p_1\cdots p_{m} \Delta
 \\ & = & \left(p_1\cdots p_{m} \Delta^{-1}\right) \left(p_1\cdots p_{m} \Delta^{-1}\right) \Delta^2
 \\ & = &  w^2 v.
\end{eqnarray*}

Hence $PC(y)=w^2v$ also when $p$ is odd. Finally, we have:
\begin{eqnarray*}
  y & = & \Delta^p (y_1\cdots y_{m}) (\tau(y_1)\cdots \tau(y_{m}))\cdots (y_1\cdots y_{m}) (\tau(y_1)\cdots \tau(y_{m}))(y_1\cdots y_{m})
\\ & = & (\tau(y_1)\cdots \tau(y_{m})) (y_1\cdots y_{m})\cdots (\tau(y_1)\cdots \tau(y_{m})) (y_1\cdots y_{m}) \Delta^p (y_1\cdots y_{m})
\\ & = & PC(y)^r \Delta^p (y_1\cdots y_{m})
\\ & = & PC(y)^r \Delta^{p+1} \Delta^{-1}(y_1\cdots y_{m})
\\ & = & PC(y)^r \Delta^{p+1} (p_1\cdots p_{m}) \Delta^{-1}
\\ & = & (w^2v)^r v^{\frac{p+1}{2}}w
\\ & = & v^{\frac{2r+1+p}{2}}w^{2r+1}
\\ & = & v^{c}w^{d},
\end{eqnarray*}
where $c=\frac{pt+2l}{2t}$ and $d=\frac{2l}{t}$ (recall that $2l=2(2r+1)m=(2r+1)t$ in this case).
\end{proof}

\section{An algorithm to find the $k$-th root of a braid}\label{section_algorithm}

We end this paper by providing a detailed algorithm that summarizes the results from the previous section, together with a study of its complexity.

The results of the previous section are valid when $USS(x)$ is minimal (which is the generic case). In order to have an algorithm which always succeeds in finding the $k$-th root of a braid~$x$, we need to include instructions on what to do if $USS(x)$ is not minimal. In those cases, one can use the algorithm in \cite{Lee}, which finds the $k$-th root of~$x$ in any case, considering the Garside group $G=\mathbb Z\ltimes \left(\mathbb B_n\right)^k$, where $\mathbb Z=\langle \delta\rangle $ acts on~$\left(\mathbb B_n\right)^k$ by cyclic permutation of the coordinates. S. J. Lee shows that the braid~$x$ has a $k$-th root if and only if the ultra summit set of $\delta (x,1,\ldots,1)$ in~$G$ has an element of the form $\delta (h,\ldots,h)$. Hence, computing an ultra summit set in such a group also solves the root extraction problem in~$\mathbb B_n$. It is not clear to us how big these ultra summit sets are in generic cases, while the algorithm presented in this paper is very simple, and generically very fast.

If one is not interested in programming the algorithm in~\cite{Lee}, one could tell our algorithm to return `fail' when $USS(x)$ is not minimal, obtaining an algorithm which will succeed only in the generic case. In any case, we present now the main result:

\begin{theorem}\label{maintheorem}
There is an algorithm that takes as input a braid $x=\Delta^p x_1\ldots,x_l \in \mathbb B_n$ written in left normal form, and a positive integer $k>1$, and finds a braid $a\in \mathbb B_n$ such that $a^k=x$, or guarantees that such a braid does not exist, whose generic-case complexity is $O(l(l+n)n^3\log n)$.
\end{theorem}

\begin{proof}
\autoref{algo}, which uses the results from the previous section, constitutes a proof of the theorem. Let us describe it in detail.

 The input is a braid $x=\Delta^p x_1\cdots x_l\in \mathbb B_n$ in left normal form and an integer $k>1$. First (lines 2-5), the algorithm applies iterated cyclic sliding to~$x$, checking at each iteration whether the resulting braid~$y$ is rigid. As we will now see, if the algorithm applies cyclic sliding $l\left(\frac{n(n-1)}{2}-1\right)$ times and no rigid braid is obtained, then we are not in the generic case stated in \autoref{teoCarusoWiest}, hence the algorithm in~\cite{Lee} is applied. The number $l\left(\frac{n(n-1)}{2}-1\right)$ is precisely $l$~times the length of~$\Delta$ minus one. Recall from \autoref{teoCarusoWiest} that in the generic case there is a positive element~$\alpha$ conjugating~$x$ to a rigid braid, such that $\ell(\alpha)< \ell(x)=l$. If~$\alpha$ is the smallest possible one, there is no~$\Delta$ in its normal form. Hence, the length of~$\alpha$ in terms of atoms ($\sigma_i$'s) is at most $l\left(\frac{n(n-1)}{2}-1\right)$. Now, from \autoref{Cyclic_Sliding_fastest_conjugation} we know that the smallest positive conjugator to a rigid braid is obtained by iterated cyclic sliding. Since at every iteration the conjugating element gets bigger, if we are in the generic case we must obtain a rigid element in at most $l\left(\frac{n(n-1)}{2}-1\right)$ iterations, as we claimed.

 If the braid~$y$ obtained after the loop in lines 2-5 is rigid, as the algorithms stores the conjugating elements for cyclic sliding at each iteration, we will have a braid~$\alpha$ such that $\alpha^{-1}x\alpha=y$.

 Now the algorithm checks whether $USS(y)$ is minimal (the generic case we are interested in), as explained in \autoref{checking_minimal}, checking whether the minimal simple elements for~$y$ are precisely~$\iota(y)$ and~$\partial(\varphi(y))$.

 In general, it is not known how fast it is to compute the minimal simple elements for a given arbitrary braid~$y$. But if~$y$ is rigid, one can easily find the minimal simple elements for~$y$. We know that every such element must be a prefix of either~$\iota(y)$ or~$\partial(\varphi(y))$. For every generator~$\sigma_i$, one can consider $\sigma_i^{-1}y\sigma_i$ and apply iterated cyclic sliding to it, until it becomes rigid. The obtained conjugating element is the smallest conjugating element from $y$ to a rigid braid, having~$\sigma_i$ as a prefix. We do this for all~$\sigma_i$ which are prefixes of~$\iota(y)$, and either we find a conjugating element which is a proper prefix of~$\iota(y)$ (in which case~$\iota(y)$ is not minimal), or we have shown that~$\iota(y)$ is minimal. Then we do the same for all generators which are prefixes of~$\partial(\varphi(y))$. The number of iterations in each case is bounded by the length  of~$\iota(y)$ (resp.~$\partial(\varphi(y))$), which are simple elements, while the total number of generators is~$n-1$. So the total number of cyclic slidings used to check whether~$\iota(y)$ and~$\partial(\varphi(y))$ are minimal (and hence whether $USS(y)$ is minimal) is~$O(n^3)$.

 If $USS(y)$ is not minimal, we are not in the generic case stated in \autoref{teoCarusoWiest}, hence the algorithm in \cite{Lee} is applied. Otherwise, we are in one of the situations described in \autoref{case1}, \autoref{case2} and \autoref{case3}. The rest of the algorithm just applies these propositions together with \autoref{prop_raiz}: after decomposing~$y$ in the form $y=v^cw^d$, it checks whether both~$c$ and~$d$ are multiples of~$k$. If this is the case, then $v^{\frac{c}{k}}w^{\frac{d}{k}}$ is the (unique) $k$-th root of~$y$, and since $x=\alpha y \alpha^{-1}$, it follows that $\alpha v^{\frac{c}{k}}w^{\frac{d}{k}}\alpha^{-1}$ is the desired $k$-th root of~$x$; otherwise, the algorithm returns the sentence ``A $k$-th root does not exist".

We study now the complexity of our algorithm, assuming that we are in the generic case in which $USS(x)$ is minimal, and we can quickly conjugate~$x$ to a rigid braid. Computing the complement or applying~$\tau$ to a simple element is~$O(n)$, and computing $s\wedge t$ for two simple elements~$s$ and~$t$ is~$O(n\log n)$ \cite[Proposition~9.5.1]{Epsteinetal}. Starting with an element~$y$ in left normal form, computing~$\mathfrak s(y)$ consists of computing a complement ($\partial(\varphi(y))$), a meet ($\iota(x)\wedge \partial(\varphi(x))$) and the normal form of the conjugate of~$y$ by a simple element of length at most~$l$ (which is $O(ln\log n)$). Hence the total complexity of applying a cyclic sliding is $O(ln\log n)$.

 The first loop (lines 2-5) is repeated~$O(ln^2)$ times, checking the condition takes $O(n\log n)$ and the body of the loop takes $O(ln\log n)$. Hence the total complexity of the loop in lines 2-5 is $O(l^2n^3\log n)$.

 The ``If" statement in lines 6-7 is negligible compared with the previous ``while" loop.

 Next, in lines 8-9 the algorithm checks whether~$\iota(y)$ and~$\partial(\varphi(y))$ are minimal, for the rigid element~$y$. By the arguments above, this applies~$O(n^3)$ cyclic slidings, hence the total complexity of this step is $O(ln^4\log n)$.

 In line 11 and in the loop in lines 12-15, some cyclings are applied. Since the involved braids are rigid of canonical length at most~$l$, and cycling is just a cyclic permutation of the factors with a possible application of~$\tau$ to a simple element, this final part of the algorithm is negligible with respect to the previous one.

 Therefore, the generic-case complexity of \autoref{algo} is $O(l(l+n)n^3\log n)$.
\end{proof}

\begin{algorithm}[ht]
\small
\SetKwData{Self}{selfConjugateOrbit}
\SetKwInOut{Input}{Input}\SetKwInOut{Output}{Output}
\Input{A braid $x\in \mathbb B_n$ given in left normal form, and an integer $k>1$.}
\Output{A braid $a\in \mathbb B_n$ such that $a^k=x$, or the message ``{\it A $k$-th root does not exist.}".}

$y:=x$;\quad $l=\ell(x)$;\quad $\alpha=1\in \mathbb B_n$;\quad $r=0\in \mathbb Z$\;
\While{$\iota(y)\wedge \partial(\varphi(y))\neq 1$ {\bf and} $r< l\left(\frac{n(n-1)}{2}-1\right)$}{
 $\alpha:=\alpha\: \mathfrak p(y)$\;
 $y:=\mathfrak s(y)$\;
 $r:=r+1$\;
}
\If{$\iota(y)\wedge \partial(\varphi(y))\neq 1$}{
  $y$ is not rigid. Apply the algorithm in \cite{Lee}\;
}
\ElseIf{$\{\mbox{Minimal simple elements for }y\}\neq \{\iota(y),\: \partial(\varphi(y))\}$}{
  $USS(y)$ is not minimal. Apply the algorithm in \cite{Lee}\;
}
\Else{

$y':=\tau(y)$; \quad $z:=\cc(y)$;\quad $PC:=\iota(y)\in \mathbb B_n$; \quad $t:=1\in \mathbb Z$;   \quad $p:=\inf(y)$; \quad $l:=\ell(y)$; \quad $\Self:=0$\;
\While{$z\neq y$ {\bf and} $z\neq y'$}
{
 $PC:=PC\: \iota(z)$\;
 $z:=\cc(z)$\;
 $t:=t+1$\;
}
\If{$z=y'$}{
   $\Self := 1$\;
 }

\If{$\Self=0$}{
 $c:= p/2$\;
 $d:= l/t$\;
 \If{$k|c$ {\bf and} $k|d$}{
  $v:=\Delta^2$\;
  $w:=PC$\;
  \Return{$\alpha v^{\frac{c}{k}}w^{\frac{d}{k}}\alpha^{-1}$\;}
 }
 \Else{
  \Return{``A $k$-th root does not exist.'';}
 }
}
\ElseIf{$\Self=1$ {\bf and} $y=y'$}{
 $c:= p$\;
 $d:= l/t$\;
 \If{$k|c$ {\bf and} $k|d$}{
  $v:=\Delta$\;
  $w:=PC$\;
  \Return{$\alpha v^{\frac{c}{k}}w^{\frac{d}{k}}\alpha^{-1}$\;}
 }
 \Else{
  \Return{``A $k$-th root does not exist.'';}
 }
}
\ElseIf{$\Self=1$ {\bf and} $y\neq y'$}{
 $t:=2t$\;
 $c:= \frac{pt+2l}{2t}$\;
 $d:= \frac{2l}{t}$\;
 \If{$k|c$ {\bf and} $k|d$}{
  $v:=\Delta$\;
  $w:=PC\: \Delta^{-1}$\;
  \Return{$\alpha v^{\frac{c}{k}}w^{\frac{d}{k}}\alpha^{-1}$\;}
 }
 \Else{
  \Return{``A $k$-th root does not exist.'';}
 }
}
}

\caption{Find a $k$-th root of a braid $x$. }\label{algo}
\end{algorithm}

\begin{remark}
Although the integers~$p$ and~$k$ are part of the input, the computed complexity does not involve them, as treating with these integers is usually negligible, in reasonable examples, with respect to the calculated complexity. If~$p$ is really big, one should take into account the number~$\log p$. The case of~$k$ is somehow different, as one would have a positive answer only if~$k$ is a divisor of the integers~$c$ and~$d$ (with $d\neq 0$), which are~$O(p+l)$, so it makes no sense to ask for a $k$-th root of~$x$, in the generic case, if~$k$ is too big compared with~$p$ and~$l$.
\end{remark}

%%%%%%%%%%%%%%%%%%%%%%%%%%%%%%%%%%%%%%%%%%%%%%%%%%%%%%%%%%%%%%%%%%%%%%%%%%%%%%%%%

%|<------------------------------------------------------------------------>|

\bibliographystyle{plain}
\bibliography{nroots}

%\begin{thebibliography}{99}

%\bibitem{PaoloPresentations} P.~Bellingeri. On presentations of surface braid groups.  J. Algebra 274 (2004), no. 2, 543-563.
%
%\bibitem{Artin} E.~Artin. Theory of braids.  Ann. of Math. 48  (1947), no. 2, 101-126.
%
%\bibitem{Birman} J. S. Birman. Braids, links, and mapping class groups. Annals of Mathematics Studies, No. 82. Princeton University Press, 1974.

%\end{thebibliography}

{\tiny IMB, UMR 5584, CNRS, UNIV. BOURGOGNE FRANCHE-COMTÉ, 21000 DIJON, FRANCE; \\
DEPARTMENT OF MATHEMATICS, HERIOT-WATT UNIVERSITY,  UNITED KINGDOM.
}\\
{\scriptsize {\it E-mail address}: M.Cumplido@hw.ac.uk}

{\tiny
DEPARTAMENTO DE \'{A}LGEBRA, UNIVERSIDAD DE SEVILLA, SPAIN.}\\
{\scriptsize{\it E-mail address}: meneses@us.es}

{\tiny
DEPARTAMENTO DE CIENCIAS INTEGRADAS, UNIVERSIDAD DE HUELVA, SPAIN.} \\
{\scriptsize{\it E-mail address}: marithania@us.es}

\end{document}